\documentclass[a4paper]{article}
\usepackage{amssymb,latexsym,amsthm,amssymb,amsfonts,amsmath,amsopn,amstext,amscd,latexsym,stmaryrd}
\usepackage{enumerate}
\usepackage{tikz}

\usepackage[mathlines]{lineno}

\newtheorem{theorem}{Theorem}[section]
\newtheorem{lemma}[theorem]{Lemma}
\newtheorem{cor}[theorem]{Corollary}
\newtheorem{proposition}[theorem]{Proposition}
\newtheorem{conjecture}[theorem]{Conjecture}
\newtheorem{corollary}[theorem]{Corollary}

\newcommand\cE{{\mathcal E}}
\newcommand\bbF{{\mathbb F}}

\def\myellipse{(0,1.2) ellipse (1.25cm and .5cm)}

\def\mytriangle{ 
      [rounded corners=2mm] (0,1)--(-0.6,2.5)--(.6,2.5)--cycle
      (-0.3, 2.2)  circle (2pt) node {}
      (0.3, 2.2)  circle (2pt) node {}
      (0, 1.5)  circle (2pt) node  {};
}

\begin{document}

\title{Infection in Hypergraphs}

\author{ 
Ryan Bergen\thanks{Department of Mathematics and Statistics,
    University of Regina, Regina, Saskatchewan, S4S 0A2. Email:
  \texttt{bergen4r@uregina.ca}.}  
\and
Shaun Fallat\thanks{Department of Mathematics and Statistics,
  University of Regina, Regina, Saskatchewan, S4S 0A2. Research
  supported in part by an NSERC research grant Application No.: RGPIN-2014-06036. Email:
  \texttt{shaun.fallat@uregina.ca}.}
\and
Adam Gorr\thanks{Department of Mathematics and Statistics, University
  of Regina, Regina, Saskatchewan, S4S 0A2. Email:
  \texttt{gorr200a@uregina.ca}.} 
\and
Ferdinand Ihringer\thanks{Department of Mathematics and Statistics,
  University of Regina, Regina, Saskatchewan, S4S 0A2 Research
  supported in part by a PIMS post-doctoral fellowship. Email:
  \texttt{Ferdinand.Ihringer@gmail.com}.} 
\and
Karen Meagher\thanks{Department of Mathematics and Statistics,
  University of Regina, Regina, Saskatchewan, S4S 0A2. Research
  supported in part by an NSERC research
  grant Application No.: RGPIN-341214-2013. \texttt{Karen.Meagher@uregina.ca}.} 
\and
Alison Purdy\thanks{Department of Mathematics and Statistics, University
  of Regina, Regina, Saskatchewan, S4S 0A2.  Email: \texttt{Alison.Purdy@uregina.ca}.} 
\and
Boting Yang\thanks{Department of Computer Science,
  University of Regina, Regina, Saskatchewan, S4S 0A2. Research
  supported in part by an NSERC Discovery Research Grant, Application
  No.: RGPIN-2013-261290.  Email: \texttt{Boting.Yang@uregina.ca}.}
\and
Guanglong Yu\thanks{Department of Mathematics and Statistics,
  University of Regina, Regina, Saskatchewan, S4S 0A2 and Department of
  Mathematics, Yancheng Teachers University, Yancheng, 224002,
  Jiangsu, P.R. China. Research supported in part by an NSFC grant
(No. 11271315), Jiangsu Qinglan Project (2014A).  Email: \texttt{yglong01@163.com}.}
}

\maketitle

\begin{abstract}
  In this paper a new parameter for hypergraphs called
  \textsl{hypergraph infection} is defined. This concept generalizes
  zero forcing in graphs to hypergraphs. The exact value of the
  infection number of complete and complete bipartite hypergraphs is
  determined. A formula for the infection number for interval
  hypergraphs and several families of cyclic hypergraphs is given. The
  value of the infection number for a hypergraph whose edges form a
  symmetric $t$-design is given, and bounds are determined for a
  hypergraph whose edges are a $t$-design. Finally, the infection
  number for several hypergraph products and line graphs are
  considered.
\end{abstract}

\section{Introduction}
\label{sec:defn}

The subject of zero forcing for graphs has been widely
studied~\cite{MR2645093, MR3010007,MR2388646, HCY}.  In this paper we
generalize the concept of zero forcing on graphs to hypergraphs.

In standard zero forcing for graphs, the vertices of the graph are
coloured either black or white. A black vertex can force a white
vertex to black according to a colour change rule. The colour change
rule for standard zero forcing is that a black vertex can force an
adjacent white vertex to black if it is the only white vertex adjacent
to that black vertex. A set of vertices in a graph is a
\textsl{zero forcing set} for the graph if when the vertices in this
set are set to black and the colour changing rule is applied
repeatedly, all the vertices of the graph are eventually forced to
black. The \textsl{zero forcing number} of a graph is the size of the
smallest zero forcing set for the graph. For a graph $G$, the
zero forcing number is denoted by $Z(G)$.

The term ``zero forcing'' is based on an algebraic property of these
sets.  Consider a vector with the entries corresponding to the
vertices of a graph. Further, assume the entries corresponding to a
set of vertices in a zero forcing set for the graph are equal to
zero. The zero forcing property of the set guarantees that such a
vector is in the kernel of the adjacency matrix of the graph only if
the vector is the zero vector. The term ``zero forcing'' refers to
the fact that the remaining entries of the vector are forced to be
zero for the vector to be in the kernel of the adjacency matrix.

For hypergraphs, there is no matrix analogous to the adjacency matrix
of a graph and this notion of a set of entries in a vector forcing the
other entries to be zero in a proposed null vector does not
apply. However, in this paper we focus on generalizing the colour
change rule and hence we use the term \textsl{infection}, rather than
zero forcing. Terms such as infection, propagation, and searching have
all been used for notions similar to zero forcing, see
\cite{propagation, severini, yang}.

For infection in a hypergraph, the vertices are initially either
infected or uninfected (as opposed to either coloured black or white,
as they are in zero forcing for graphs). There is an \textsl{infection
  rule} that determines when vertices can infect other vertices (this
is analogous zero forcing rule for graphs). In this case, it is a
subset of infected vertices in an edge that may infect the remaining
vertices in that specific edge, rather than a single vertex forcing
another vertex.  The following is the infection rule for hypergraphs.

{\bf Infection Rule:} A non-empty set $A$ of infected vertices
can infect the vertices in an edge $E$ if:
\begin{enumerate}
\item $A \subset E$, and 
\item there are no uninfected vertices $v$, not contained in $E$, such that $A
  \cup \{v\}$ is a subset of an edge.
\end{enumerate}

Similar to the case for graphs, if two vertices in a hypergraph are
contained in a common edge, then we say that the vertices are
\textsl{adjacent}. Further, for a hypergraph, we can define two sets to
be adjacent if there is an edge that contains them both. So a set $A$
of infected vertices can infect an edge $E$ if there are no
uninfected vertices outside of $E$ that are adjacent to $A$.

If $A \subset E$ satisfies the conditions set out in the infection
rule, then we say that ``the set $A$ infects the edge $E$''. In the
case of hypergraphs, it is an edge and all the vertices in the edge
that are infected rather than a single vertex, as is the case for
graphs. A set of vertices in a hypergraph is an \textsl{infection set}
if when the vertices in the set are initially infected and the
infection rule is applied repeatedly, then all the vertices in the
hypergraph become infected. The \textsl{infection number} of a
hypergraph $H$ is the size of a smallest infection set for $H$; the
infection number of $H$ is denoted by $I(H)$.

In a hypergraph the edges are subsets of the vertices, and can be of
any size (the \textsl{size} of an edge is the number of
vertices in the edge). If all the edges in a hypergraph contain
exactly $k$ vertices, then the hypergraph is called a
\textsl{$k$-hypergraph}. A 2-hypergraph is a graph; the infection
number for a 2-hypergraph is equivalent to the zero forcing number of
the graph.

\begin{proposition}
Let $H$ be a $2$-hypergraph, then $Z(H) = I(H)$.
\end{proposition}
\begin{proof}
  In a 2-hypergraph, the condition that a vertex $v$ can apply a force
  is equivalent to the condition that a subset of vertices (which in
  this case is a singleton) can infect an edge.
\end{proof}

Let $H$ is a hypergraph and $W$ a subset of the vertices of $H$. The
set of all vertices in $H$ that are infected after repeatedly applying the
infection rule, with $W$ being the set of initially infected vertices,
is called the \textsl{derived set of $W$}. This is denoted by $I_W$. A
set $W$ is an infection set if and only if the derived set is the set
of all vertices.

The empty hypergraph is the hypergraph with no vertices and no edges, we
will not consider this case. A \textsl{trivial hypergraph} is
hypergraph with vertices, but no edges. The infection number for any
trivial hypergraph is clearly the number of vertices in the hypergraph
since no set can ever infect any edge. For every other hypergraph,
there is an upper bound on the size of the infection number for a
hypergraph.

\begin{proposition}\label{prop:upperbound}
  Let $H$ be a non-trivial hypergraph on $n$ vertices and let $k$ be the size of
  the largest edge in $H$, then
\[
I(H) \leq n-k+1.
\]
\end{proposition}
\begin{proof}
  Let $A$ be a $(k-1)$-subset of an edge $E$ of size $k$ in the
  hypergraph. We claim that the set of all vertices except the
  vertices in $A$, form an infection set of size $n-(k-1)$ for the
  hypergraph. This follows since the final element in $E$ can infect
  $E$, since it cannot be adjacent to any uninfected vertices outside
  of $E$, as there are no uninfected vertices outside of $A$.
\end{proof}
 
In Section~\ref{sec:completes}, we will see that this bound holds with
equality for the complete hypergraph; we will also demonstrate other hypergraphs
where this bound is tight.

The \textsl{line graph} of a hypergraph $H$ is the graph formed by
representing each edge of $H$ by a vertex; these vertices are adjacent
in the line graph if and only if the corresponding edges of the
hypergraph $H$ intersect. The line graph of a hypergraph $H$ is
denoted by $L(H)$, we consider these graphs in more detail in
Section~\ref{sec:linegraph}. A hypergraph is \textsl{connected} if and
only if its line graph is connected, that is a hypergraph is connected if
there is a path between any two vertices in the line graph. A
\textsl{connected component} of a hypergraph is a maximal connected
sub-hypergraph. It is not hard to see that the infection number of a
hypergraph is the sum of the infection numbers of the connected
components of the hypergraph.

\begin{proposition}\label{prop:connected}
If $H$ is a hypergraph with connected components $H_1, \dots, H_t$ then
\[
I(H) = \sum_{i=1}^t I(H_i).
\]
\end{proposition}

\section{Reduced hypergraphs}

One significant difference between graphs and hypergraphs is that
there is no restriction on the size of an edge in a hypergraph, where
in a graph all edges have size two. This can lead to the situation
where one edge of a hypergraph is a subset of another edge (a
hypergraph that does not have this property is called
\textsl{reduced}). The most extreme case of this is when the entire
set of vertices forms an edge.

\begin{proposition}\label{prop:superedge}
Let $H$ be a hypergraph with vertex set $V$. If $V$ is an edge of $H$,
then $I(H) = 1$,
\end{proposition}
\begin{proof} 
  This follows simply from the fact that any one vertex can infect the
  edge containing all the vertices.
\end{proof}

Consequently, the infection number of a hypergraph is not monotone
under sub-hypergraphs; indeed, simply adding an edge that contains the
entire vertex set of the hypergraph will reduce the infection number
to one.

The next result shows that we can remove any edge that is a
subset of another in a hypergraph without changing the infection
number.

\begin{proposition}
  Let $H$ be a hypergraph and assume that $E_1$ and $E_2$ are edges of
  $H$ with $E_1 \subset E_2$, then $I(H) = I(H\backslash E_1)$.
\end{proposition}
\begin{proof} 
  Assume that $A$ is an infection set for $H$.  If $A' \subseteq A$
  can infect an edge $E$ (with $E_2 \not \subseteq E$), then $A'$ can
  also infect $E$ in $H\backslash E_1$.  So starting with $A$, the
  infection process in $H\backslash E_1$ can progress as it does in
  $H$.  At some point in the infection process for $H$, a set $A'$
  infects the edge $E_1$.  This set $A'$ can also infect $E_2$, since
  all of the vertices in $E_2\setminus E_1$ must already be
  infected. Thus $A$ forms an infection set for $H\backslash E_1$, and
  hence $I(H) \geq I(H\backslash E_1)$.

  Conversely, assume that $A$ is an infection set for $H\backslash
  E_1$.  Assume that $A' \subseteq A$ can infect an edge $E$, then $A'$
  is not adjacent to any uninfected vertices in $E_2 \backslash E$,
  and thus $A'$ is also not adjacent to any uninfected vertices in
  $E_1 \backslash E$. So $A'$ can infect $E$ in the hypergraph $H$,
  and the infection process in $H$, starting with $A$ can progress as
  it does in $H\backslash E_1$. Further, $E_2$ is infected at some
  point in the infection process in $H\backslash E_1$, at this step in
  $H$, the edge $E_1$ is infected.  Thus $A$ is also an infection set
  for $H$ and hence $I(H) \leq I(H\backslash E_1)$. Thus $I(H) =
  I(H\backslash E_1)$.
\end{proof}

This proof implies not only that $I(H) = I(H\backslash E_1)$, but also that
an infection set for one of the hypergraphs is also an infection set
for the other.

Many of the notations and concepts for graphs can be generalized to
hypergraphs. For a hypergraph $H$, the vertex set is denoted by $V(H)$
and the edge set by $E(H)$. If $v$ is a vertex in a hypergraph $H$,
then the \textsl{degree} of $v$, denoted $\deg(v)$, is the number of
edges that contain $v$. Further, if $A \subseteq V(H)$, then the
degree of the set $A$ is defined by
\[
\deg(A) = |\{E \in E(H) : A \subseteq E\} |.
\]
Similar to the notation for standard graphs, we will use $\delta(H)$
to denote the minimum degree of a vertex in a hypergraph $H$. Any
subset that has degree $1$ can infect the edge that contains it. But
Proposition~\ref{prop:superedge} shows that a vertex of high degree can
still infect an edge. 

A hypergraph is said to be \textsl{reduced} if no edge is the subset
of another edge. A hypergraph can be reduced by simply removing all
edges that are contained in another edge. If a hypergraph is not
reduced, then it is possible that the infection number is equal to
one, while the minimum degree is very large. This is not the case for
reduced hypergraphs.  

\begin{lemma}\label{lem:degreebound}
If $H$ is a reduced hypergraph and $I(H) =1$, then $\delta(H) = 1$.
\end{lemma}
\begin{proof}
  Assume that $H$ is a reduced hypergraph with $I(H) =1$. Let $\{v\}$
  be an infection set of size one and let $E$ be the first edge that
  $v$ infects. Then $v$ cannot be adjacent to any vertex outside of
  $E$. Since the hypergraph is reduced, this implies that $v$ is in no
  other edges and hence has degree one.
\end{proof}

This lemma also implies that in a reduced hypergraph with $I(H) =1$, the vertex
that does the first infection must have degree one. It is not hard to
see that the minimum degree in a graph is a lower bound for the
zero forcing number. It is not true in general that $\delta(H)$ is
a lower bound for $I(H)$ for a hypergraph $H$.
Lemma~\ref{lem:complete} in the next section shows that the complete
hypergraph is such an example where the minimum degree of a hypergraph is
much larger than the infection number.

Throughout the remainder of this paper, we will assume that our
hypergraphs are \textsl{reduced}.

\section{Complete hypergraphs and complete bipartite hypergraphs}
\label{sec:completes}

Let $H$ be a non-empty $k$-hypergraph on $n$ vertices. Then, from
Proposition~\ref{prop:upperbound}, the infection number of $H$ is no
more than $n-k+1$. In this section, we consider some $k$-hypergraphs on
$n$ vertices with infection number $n-k+1$.

The \textsl{complete $k$-uniform hypergraph}, denoted by $H^{(k)}_n$, has all
$k$-subsets of $\{1,\dots,n\}$ as its edges.

\begin{lemma}\label{lem:complete}
For any $k$ and $n$ with $1\leq k \leq n$, the value of $I(H^{(k)}_n)$ is $n-k+1$.
\end{lemma}
\begin{proof}
  From Proposition~\ref{prop:upperbound} it follows that the infection
  number is no more than $n-k+1$. In fact, for this graph any set of
  $n-k+1$ vertices is an infection set.

  To see that the infection number cannot be any smaller, assume that
  there are initially $k$ uninfected vertices in $H^{(k)}_n$, say
  $v_1, \dots, v_k$.  Let $A$ be any subset of the $n-k$ infected vertices that
  can infect some edge $E$. Then $E = A \cup B$ where $B$ is a proper
  subset of $\{v_1, \dots, v_k\}$.  We can assume that $v_1 \in
  B$ and $v_k \not \in B$. Since the hypergraph is the complete hypergraph, the
  edge $A \cup (B\backslash \{v_1\} \cup \{v_k\})$ is an edge in the
  hypergraph. This means that $v_k$ is an uninfected vertex that is
  adjacent to $A$ outside of $E$. This is a contradiction, so no set
  of fewer than $n-k+1$ vertices can infect an edge.
\end{proof}

The complete hypergraph $H_n^{(2)}$ is isomorphic to the complete graph
$K_n$. The previous result shows that $I(H^{(2)}_n) = Z(K_n) = n-1$.
In a graph $G$, a set of vertices in which any two are adjacent is
called a \textsl{clique}; the complete graph is a clique. The size of
the maximum clique in $G$ is usually denoted by $\omega(G)$ and the
value of the zero forcing number is bounded below by $\omega(G) -1$~\cite{MR2388646}.

This is another point where a major difference between the zero
forcing number and the infection number is apparent. There is
no comparable bound for a hypergraph that contains a subgraph isomorphic
to the complete hypergraph $H_n^{(k)}$. For example, adding the edge
$\{1,\dots,n\}$ to the complete hypergraph $H_n^{(k)}$ produces a
hypergraph that contains $H_n^{(k)}$ as a sub-hypergraph, but has
infection number $1$. Next we will see how to construct a
$k$-hypergraph that has $H_n^{(k)}$ as a sub-hypergraph, and infection
number equal to $1$.

\begin{proposition}
  Let $H$ be any $k$-hypergraph. There exists a $k$-hypergraph $H'$ such
  that $E(H) \subset E(H')$ and $I(H') =1$.
\end{proposition}
\begin{proof}
 Assume that the vertex set
  for $H$ is $\{1,\dots,n\}$. The construction of the hypergraph $H'$ proceeds as follows.
  First, $\{1,\dots,n\}$ will be
  vertices in $H'$, and all the edges of $H$ will be edges of $H'$.

  Second, the vertices of $H$ can be covered by $r_1= \lceil \frac{n}{k-1}
  \rceil$ sets each of size $k-1$. Call these sets $A_1, A_2, \dots
  A_{r_1}$. Add the vertices $\{n+1, \dots,
n+r_1\}$ to the vertex set of $H$ and add sets
\[
A_1 \cup \{n+1\}, A_2 \cup \{n+2\}, \dots, A_{r_1} \cup \{n+r_1\}
\]
to the edge set of $H'$.

Third, the vertices $\{n+1, n+2, \dots, n+r_1\}$ can be covered in $r_2=
\lceil \frac{r_1}{k-1} \rceil$ sets each of size $k-1$. Call these
sets $B_1, B_2, \dots, B_{r_2}$, and add the sets
\[
B_1 \cup \{n+r_1+1\}, B_2 \cup \{n+r_1+2\}, \dots, B_{r_2} \cup \{n+r_1+r_2\}
\]
to the edge set of $H'$, and include the additional vertices $\{n +
r_1+1, \dots, n+r_1+r_2\}$ in $H'$. 

Continue in this fashion until $r_j = 1$. The final vertex to be added
to the hypergraph $H'$ is $n+ r_1 +r_2 + \dots + r_j$. This final
vertex is an infection set for the $k$-hypergraph $H'$.
\end{proof}

Not only are there hypergraphs with infection number $1$ that contain
a complete hypergraph, but there are also $k$-hypergraphs on $n$
vertices with infection number $n-k+1$ which are not complete
hypergraphs.

\begin{proposition}\label{prop:addvertextocomplete}
  For $2 \leq k \leq n$, let $H$ be the $k$-hypergraph on $n$ vertices that is formed by
  adding the element $n$ to each of the edges in $H_{n-1}^{(k-1)}$.
  If $n \geq 2k-1$, then the infection number of $H$ is $n-k$; otherwise, it
  is $n-k+1$.
\end{proposition}
\begin{proof}
  First consider the case where $n \geq 2k-1$.  Assume that the
  vertices labeled $1, \dots, n-k$ are infected and the vertices
  labeled $n-k+1, \dots, n$ are not infected.  Thus there are $k$
  uninfected vertices including $n$.  Since $n \geq 2k-1$ there are at
  least $k-1$ infected vertices, and any $(k-1)$-set of infected
  vertices can infect the vertex $n$. Once $n$ is infected, there are
  only $k-1$ uninfected vertices remaining.  By construction, they
  will be in an edge with $n$ and thus can all be infected by
  $n$. Thus $I(H) \leq n-k$.

  To see that the infection number cannot be less than $n-k$, let $A$
  be any set of infected vertices of size $n-k-1$ and assume that the
  remaining $k+1$ vertices are uninfected. Thus there must be at least
  $k$ vertices from $\{1,\dots,n-1\}$ that are uninfected. An argument
  similar to that used in the proof of Lemma~\ref{lem:complete} will
  show that $A$ cannot be an infection set.

  Next consider the case where $n \leq 2k-2$.  Since $n-k+1$ is always
  an upper bound on the infection number, we only need to prove that
  no set of $n-k$ vertices can be an infection set.  If $n=k$ this is
  trivial, so assume that $n > k$.  Let $A$ be a set of $n-k$
  infected vertices and assume the remaining $k$ vertices are
  uninfected.  Since $|A| \leq k-2$, for every edge
  containing an uninfected vertex $v \in \{1,\dots,n-1\}$ and a subset
  $A'$ of $A$, there will be at least one uninfected vertex $w \in
  \{1,\dots,n-1\}$ not in the edge.  By construction of $H'$, there will be
  another edge containing $A'$ and $w$, and thus the set $A$ can never
  infect an edge.
\end{proof}

The zero forcing number of the complete bipartite graph $K_{n_1,n_2}$
is known to be $n_1+n_2-2$. We derive a similar result for
complete $k$-partite hypergraphs.  Define the \textsl{$k$-partite
  complete hypergraph} $H^{(k)}_{n_1,n_2,\dots, n_k}$ as follows. The
vertex set of $H^{(k)}_{n_1,n_2,\dots, n_k}$ is $V$, which can be
partitioned into $k$ disjoint parts, namely $V_1, V_2, \dots, V_k$, where
$|V_i| = n_i$.  The edge set is the set of all $k$-sets with exactly
one element from each of $V_i$ where $i = 1, \dots, k$.

\begin{lemma}
If $H = H^{(k)}_{n_1,n_2,\dots, n_k}$ is a $k$-partite complete
hypergraph, then $I(H) = n_1+n_2+ \dots +n_k -k$.
\end{lemma}
\begin{proof}
  If we choose one vertex in each of $V_i$ for $i = 1, \dots ,k$ to be
  uninfected and the remaining vertices to be infected, then we have
  an infection set of the appropriate size.

  Similar to the case for the complete hypergraph (see the proof of
  Lemma~\ref{lem:complete}), if there are two uninfected vertices in
  some $V_i$, then no set can infect either of these vertices.
\end{proof}

The complete bipartite graph $K_{1,n}$ is also called a
\textsl{flower} or a \textsl{star}. For this graph, the intersection of any two edges is
the fixed single vertex in the first partition. This structure can be generalized
to a hypergraph. If $H = \{E_1, E_2, \dots, E_n\}$ is a hypergraph
with the property that for any two distinct $i, j \in\{1, \dots, n\}$ the
intersection $E_i \cap E_j = A$ (where $A$ is a non-empty set), then
$H$ is called a \textsl{flower}. The sets $E_i$ are called the petals.

\begin{lemma}
  Let $H$ be a hypergraph that is a flower with $p$ petals. Then $I(H) =
  p-1$. 
\end{lemma}
\begin{proof}
  An infection set can be formed by taking one vertex of degree 1 from
  all but one of the petals. Each of these vertices can infect the
  edge that contains it. Then any vertex in the intersection of
  all the edges can infect the final edge.
\end{proof}

\section{Interval hypergraphs and cyclic interval hypergraphs}

It is well-known for a graph $G$ that $Z(G) = 1$ if and only if $G$ is
a path, and if $G$ is a cycle, then $Z(G) = 2$ (see \cite{MR2388646}).
In this section we will consider hypergraphs that are analogous to
paths and cycles.

Paths and cycles both have the property that the degree of every
vertex is no more than $2$. For hypergraphs there is an analogous
bound on the infection number.

\begin{lemma}
  Let $H$ be a connected $k$-hypergraph in which every vertex has
  degree no more than $2$, then $I(H) \leq k$.
\end{lemma}
\begin{proof}
  We claim that any edge of $H$ is an infection set. Let $E =\{v_1,
  \dots ,v_k\}$ be any edge in $H$, and assume that $v_1, \dots , v_k$
  are all initially infected. Each $v_i$ is contained in only one edge, other
  than $E$, and hence can infect that edge. The same now holds for all
  other vertices in the newly infected edges.  Since $H$ is connected,
  all vertices will be infected under this process.
\end{proof}

Another way to generalize paths to hypergraphs is by \textsl{linear
  hypergraphs}.  A hypergraph $H$ is linear if $|E_i \cap E_j| \leq 1$
for any two edges $E_i$ and $E_j$ of $H$. This implies that any set of
two of more vertices can be in at most one edge, so a set of more
than one vertex will have degree at most one.

\begin{lemma}
  If $H$ is a reduced connected, linear hypergraph and all the
  vertices have degree no more than two, then $I(H) \leq 2$. Further,
  $I(H)=1$ if and only if $H$ has a vertex of degree one.
\end{lemma}
\begin{proof}
  Let $H$ be a connected, linear hypergraph in which all vertices have
  degree no more than $2$.  Since $H$ is linear, any two adjacent
  vertices are contained in exactly one edge; thus any two adjacent
  vertices can infect the edge that contains them. Since $H$ is
  connected, starting with any two adjacent vertices being initially
  infected, the infection process will infect every vertex in the
  hypergraph.

  If $H$ has a vertex with degree one, then this vertex can infect the
  edge that contains it. Then the vertices in this edge are each in at
  most one other edge, which they can now infect. Continuing like this
  infects all the vertices of the hypergraph. The converse holds by
  Lemma~\ref{lem:degreebound}.
\end{proof}

Next we consider a different generalization of paths in which the 
degrees of the vertices can be more than two.  An \textsl{interval ordering} of
the vertices of a hypergraph is an ordering of the vertices so that
every hyperedge of the hypergraph is an interval of the ordering. We
say that a hypergraph is an \textsl{interval hypergraph} if there
exists an interval ordering of the vertices of the hypergraph.

\begin{lemma}
If $H$ is a reduced interval hypergraph, then $I(H)$ is equal to the number
of connected components.
\end{lemma}
\begin{proof}
  Since $H$ is an interval hypergraph, there is a linear ordering of
  the vertices and this ordering extends to the edges ($E_1 < E_2$ if
  this minimal element of $E_1$ is smaller than the minimal element of
  $E_2$). 

  Let $v_1$ be the first vertex in the ordering, and suppose $v_1 \in
  E$. Since $H$ is reduced, it follows that $v_1$ has degree one, and
  is not adjacent to any other vertices outside of $E$. Thus $v_1$ can
  infect $E$.  Let $v_2$ be the least vertex in the ordering that is
  not infected. Similar to $v_1$, this vertex can infect the largest
  edge that contains it. Continuing like this will eventually infect
  all the vertices in the connected component that contains $v$. The
  result then follows from Proposition~\ref{prop:connected}.
\end{proof}

To generalize a cycle, we define a \textsl{hypercycle} to be a set of
edges $E_1,E_2, \dots, E_k$ with $E_i \cap E_{j} \neq \emptyset$ if and
only if $|i-j| \equiv 1 \pmod {k}$. In a hypercycle the maximum degree
of a vertex is two.

\begin{proposition}
  Let $H$ be a hypercycle, then $I(H) \leq 2$. Further, $I(H)=1$ if
  and only if $H$ has a vertex of degree $1$.
\end{proposition}
\begin{proof}
If $H$ does not have a vertex of degree one, then pick one vertex in
$E_1 \cap E_2$ and one vertex in $E_2 \cap E_3$. These two vertices
can infect the edge $E_2$. Then the vertices in $E_2 \cap E_3$ can
infect $E_3$. Continuing like will infect all the remaining vertices in
the hypergraph. By Lemma~\ref{lem:degreebound}, $I(H)=2$, unless $H$
has vertex of degree one.

If $v \in E_i$ is a degree one vertex, then $v$ can infect $E_i$.  Then the
remaining vertices in $H$ can be infected as before, so $I(H) =1$.
\end{proof}

The vertices of a hypergraph have an \textsl{arc-ordering} if the
vertices can be cyclically ordered so that every edge is an
$\textsl{arc}$, that is a set of vertices that are consecutive in this
ordering. A hypergraph is a \textsl{circular-arc hypergraph} if there
exists an arc-ordering of the vertices in the hypergraph.

One example of a circular-arc hypergraph is what we call the
\textsl{$t$-tight $k$-uniform cycle on $n$-vertices}. This hypergraph is
denoted by $C^{(k)}_{n}(t)$, and for it to be well-defined, we must
assume that $k-t$ divides $n$.  The vertex set of $C^{(k)}_{n}(t)$ is $\{1,
\dots, n\}$ and the following is the hyperedge set:
\begin{align*}
E_1 &=  \{1, 2, \dots, k\},  \\
E_{k-t+1} &= \{k-t+1, k-t+2, \dots, 2k-t\},  \\
E_{2k-2t+1} &= \{2k-2t+1, 2k-2t+2, \dots, 3k-2t\}, \\
& \vdots  \\
E_{n-k+t} &=  \{n-k+t, \dots, t-1,t\} \\
\end{align*}
(entries taken modulo $n$). For $i = 0, \dots, \frac{n}{k-t} - 1$,
the edge $E_{i(k-t)+1}$ is an arc that starts with the vertex
$i(k-t)+1$. The function $f_i:V(C^{(k)}_{n}(t)) \rightarrow
V(C^{(k)}_{n}(t))$ defined by $f_i(v) = v + i(k-t)$ (modulo $n$) is an
automorphism of the hypergraph that maps edge $E_j$ to
$E_{i(k-t)+j}$. So $C^{(k)}_{n}(t)$ is edge transitive.

We consider the special case when $t = k-1$; this is the largest
possible value of $t$, and $k-t = 1$, so such a hypergraph is
well-defined for all values of $k$ and $n$. We start with a simple
lemma about the structure of this graph.

\begin{lemma}\label{lem:any2in2}
  If $k+1 \leq n < 2k-1$, then any pair of vertices from
  $C^{(k)}_{n}(k-1)$ is contained in at least two edges.
\end{lemma}
\begin{proof}
  Let $x$ and $y$ be two vertices with $x<y$. If $y-x \leq n-k$, since
  $n-k\leq k-1$, then $x \notin E_{y}$, but both $x$ and $y$ will be
  in the edges $E_{y-k+1},E_{y-k+2},\dots,E_{x}$ (subscripts taken
  modulo $n$).  Conversely, if $y-x >n-k$, then both $x$ and $y$ are
  in the edges $E_{x-k+1},\dots, E_{y}$. This shows that for these
  values of $n$ and $k$ any pair of vertices will be in at least two
  edges.
\end{proof}

Next we give the exact value of the infection number for
$C^{(k)}_{n}(k-1)$.

\begin{lemma}
 Let $H= C^{(k)}_{n}(k-1)$.
\begin{enumerate}
\item If $n \geq 2k-1$, then $I(H)=2$.
\item If $k+1 \leq n < 2k-1$, then $I(H)=\min\{i+1, n\!-\!k\!+\!1\}$ where
\[
i = \left \lceil \frac{k-1}{n-k} \right \rceil .
\]
\item If $n=k$, then $I(H) = 1$.
\end{enumerate}
\end{lemma}
\begin{proof}

  Statement (3) is trivial, since if $n= k$, the hypergraph has
  only one edge. So we consider Statements (1) and (2).

  Observe that once one edge is infected, the remainder of the
  hypergraph can be infected.  For example, if $E_{1}$ is infected,
  then the set $\{2, \dots, k\}$ will infect $E_{2}$.  Subsequently,
  $\{3, \dots, k+1\}$ will infect $E_{3}$ and so on.

  If $n \geq 2k-1$, then the set $A = \{1,k\}$ is a subset of $E_1$ and
  no other edge. Thus if $A$ is initially infected, it can infect the
  remaining vertices in $E_1$.  Finally, since no vertex is contained
  in only one edge, the infection number cannot be one.  This proves statement (1).

  Assume that $k+1 \leq n < 2k-1$. To prove statement (2), we
  first show that infection sets of size $i+1$ and $n-k+1$ are
  possible for all $n$ and $k$.  

  As shown in Proposition~\ref{prop:upperbound}, an infection set of
  size $n-k+1$ can be constructed for any non-empty hypergraph. To see
  that an infection set of size $i+1$ can be constructed for $H=
  C^{(k)}_{n}(k-1)$, let 
\[ 
i = \left \lceil \frac{k-1}{n-k} \right  \rceil
\]
and consider the subset 
\[
A=\left\{1, \,  n\!-\!k\!+\!1,\,  2(n\!-\!k)\!+\!1, 
  \dots, (i-1)(n\!-\!k)\!+\!1, \, k\right\} \subset E_1.
\]

If $v_{1}$ and $v_{2}$ are vertices in $A$ such that $v_{2}$ follows
$v_{1}$ in the order shown, then $v_{2} - v_{1} \leq n-k$ and so
$v_{1} \notin E_{v_{2}}$.  Thus $A$ only occurs as a subset of $E_{1}$
and so can infect $E_{1}$.

We now show by contradiction that an infection set of size less than
$\min\{i+1, n-k+1\}$ is not possible. Assume that $B$ is an infection
set of size less than $ \min\{i+1, n-k+1\}$. Then $|B| \leq i $ and
$|B| \leq n-k$ and there are at least $k$ uninfected vertices.  Since
$H$ is edge transitive, we can assume that $B$ will first infect the
edge $E_{1}$.  Let $B' = B \cap E_{1}$.  Since $B$ infects $E_{1}$,
there will be no uninfected vertices outside of $E_{1}$ that are
adjacent to $B'$.

Suppose $B'$ contains only one element, $x$.  Then there is at least
one uninfected vertex not in $E_{1}$ and, since $n < 2k-1$, by
Lemma~\ref{lem:any2in2} it will be in an edge with $x$ other than
$E_1$.  This is a contradiction, so $B'$ must contain more than one
vertex.

Let $B' =\{v_{1},\dots,v_{\ell}\}$ where $v_{1} < v_{2} < \cdots <
v_{\ell}$ and $2 \leq \ell \leq \min\{i, n-k\}$. If $v_{(j+1)} - v_{j}
> n-k$ for any $j \in \{1,\dots,\ell-1\}$, then $v_{j} \in
E_{v_{(j+1)}}$ from which it follows that $B' \subseteq
E_{v_{(j+1)}}$. There are at least two uninfected vertices that are
not in $E_{1}$. These vertices will be in $E_{v_{(j+1)}}$ and thus will
be adjacent to $B'$.  This contradicts the assumption that $B'$
infects $E_1$. Thus we must have $v_{(j+1)} - v_{j} \leq n-k$ for all
$j \in \{1,\dots,\ell-1\}$.

Let $C= \{v_{1}, v_{1}+1, v_{1}+2, \dots, v_{\ell}\}$.  In other
words, $C$ is the interval of length $v_{\ell} - v_{1} +1$ starting
with $v_{1}$. Note that $B' \subseteq C$.  The vertices in the edges
either starting with $v_1$ or ending with $v_\ell$ are all adjacent to
$C$. If $n$ is sufficiently small, these are all vertices in the
hypergraph (outside of $C$), otherwise they are a set of vertices of
size $2(k-(v_{\ell} - v_{1} +1))$.  Thus the number of vertices
adjacent to $C$ (but not in $C$) is equal to
\[
\min\{n-(v_{\ell} - v_{1} +1),\, 2(k-(v_{\ell} - v_{1} +1))\}.
\] 
Since $k-(v_{\ell} - v_{1} +1)$ of these vertices are in $E_{1}$, the
number of vertices not in $E_{1}$ that are adjacent to $C$
is $$\min\{n-k, k-(v_{\ell} - v_{1} +1)\}.$$ The number of infected
vertices adjacent to $C$ and not in $E_{1}$ is at most $|B| - \ell$.
We now show that this number is less than the total number of vertices
adjacent to $C$ and not in $E_{1}$. This proves that there is an
uninfected vertex, not in $E_1$, that is adjacent to $B'$, which is a
contradiction.

 Suppose that the number of vertices adjacent to $C$ and not in
 $E_{1}$ is $n-k$ (i.e. $n-k \leq k- (v_{\ell} - v_{1} +1)$). Since
 $|B| \leq n-k$ and $\ell \geq 2$, it follows that $|B| - \ell < n-k$
 and so there is an uninfected vertex outside of $E_{1}$ that is
 adjacent to $B'$.

 Now suppose that $k-(v_{\ell} - v_{1} +1) < n-k$.  It follows from
 $v_{(j+1)} - v_{j} \leq n-k$ for all $j \in \{1,\dots, \ell-1\}$ that
 $v_{\ell} - v_{1} \leq (\ell-1)(n-k)$.  Therefore, 
\[
k-(v_{\ell} - v_{1} +1) \geq k-1- (\ell-1)(n-k) = \left(\frac{k-1}{n-k}\right)(n-k)
 - (\ell-1)(n-k).
\]
By definition $i-1 < \frac{k-1}{n-k}$, 
\[
k-(v_{\ell} - v_{1} +1) > (i-\ell)(n-k) \geq (|B|-\ell)(n-k) \geq |B| - \ell
\] 
and again there is an uninfected vertex not in $E_{1}$ that is
adjacent to $B'$.
\end{proof}

We consider one other special case, namely when $n=k+1$; this is the
smallest non-trivial value for $n$. In this case, the number of edges
is $\frac{n}{k-t}$, and each edge misses exactly one element.

\begin{proposition}
  Let $t,k$ and $n$ be integers and assume that $k-t$ divides $k+1$.
  The infection number of $C_{k+1}^{k}(t)$ is equal to $\frac{k+1}{k-t} - 1$.
\end{proposition}
\begin{proof}
  To construct an infection set of this size take the vertex missing
  from every edge except one (call this edge $E$).  These vertices
  only occur together in the edge $E$, so they can infect $E$. This
  leaves only one uninfected vertex (the vertex missing from $E$) which
  can be infected by any of the vertices in $E$.

  Let $A$ be a set of size $\frac{k+1}{k-t} - 2$. Each of the
  $\frac{k+1}{k-t}$ edges does not include a single vertex. Thus are
  there at least two edges that miss the vertices that are not contained
  in the set $A$. So $A$ must be a subset of both of these edges,
  and hence cannot infect them.
\end{proof}

\section{Hypergraphs from $t$-designs}

For a hypergraph $H$ there must be a set of size no more than $I(H)$
that causes the first infection. For this to be possible, this set
must be contained in exactly one edge.  Conversely, if every subset of
vertices of size $t$ is contained in at least two edges, we can
conclude that the infection number is strictly larger than $t$. This
observation is useful in determining the infection number of
hypergraphs that are themselves combinatorial designs.

A $2$-$(n,k,1)$ design is a collection of $k$-sets (called
\textsl{blocks}) from the base set $\{1, \dots, n\}$ such that each
pair from the base set occurs in exactly one $k$-set. The blocks of
any $2$-$(n,k,1)$ design form a linear hypergraph with
$\frac{n(n-1)}{k(k-1)}$ edges.  If a $2$-$(n,k,1)$ design has
$n=k^2-k+1$, then the design is a \textsl{symmetric design}; in this
case any two blocks intersect. Further, given a $2$-$(n,k,1)$ design,
a sub-design is simply a collection of blocks from the design that
forms a $2$-$(m,k,1)$ design for some $m <n$. For more on designs
see~\cite{MR2246267}, or any standard reference on design theory.

\begin{proposition}
  Let $H$ be a hypergraph in which the edges form a symmetric design
  with parameters $2$-$(k^2-k+1,k,1)$. If $k\geq 3$, then the
  infection number of $H$ is $3$.
\end{proposition}
\begin{proof}
  Assume the vertices of $H$ are $\{1, \dots, n\}$ (and $n =
  k^2-k+1$). Each pair of vertices is contained in exactly one edge of
  $H$; thus any pair can infect the edge that contains it. If $k>2$
  and all the vertices in a single edge are infected, and no other
  vertices are infected, then the process stops. Thus the infection
  number of $H$ must be greater than $2$.

  Assume, without loss of generality, that $E=(1,2, \dots, k)$ is a
  block in $H$.  We claim that the set $\{1,2, k+1\}$ (or any triple,
  not contained in a single block) can infect the hypergraph.

  First the vertices $1,2$ can infect the edge $E$. At this stage there are
  $k+1$ vertices that are infected. Then for any $i \in \{1, \dots,
  k\}$ the pair $i, k+1$ can infect the edge that contains both $i$
  and $k+1$; call this edge $E_i$. For distinct $i,j$ the intersection
  of $E_i$ and $E_j$ is exactly $k+1$ (since $H$ is linear). So at
  this stage $k(k-2)$ new vertices are infected. Since $k(k-2) + k+1 =
  k^2-k+1 =n$, all the vertices in $H$ are infected.
\end{proof}

The natural generalization of $2$-$(n,k,1)$ designs are $t$-$(n,k,1)$
designs.  A $t$-$(n,k,1)$ design is a set of $k$-subsets (called
blocks) from $\{1, \dots, n\}$ with the property that every $t$-subset
occurs in exactly one of the $k$-subsets.  We will consider
hypergraphs in which the edge set forms a $t$-$(n,k,1)$ design. Note
that any $t$-set, will occur in a single block, and any $(t-i)$-set
will occur in exactly $\binom{n-(t-i)}{i}/\binom{k-(t-i)}{i}$
blocks. Just as for $2$-designs, the infection number for a hypergraph
whose edges form a $t$-$(n,k,1)$ design must be at least $t$.

In this case, we need to consider sub-hypergraphs.  If $H$ is a
hypergraph and $W$ is a subset of vertices in $H$, then the
\textsl{sub-hypergraph induced by $W$} is the hypergraph with vertex
set $W$ and the edges are exactly the edges from $H$ which only
contain vertices from $W$.

\begin{theorem}\label{thm:derivedmakessubdesign}
  Let $H$ be a hypergraph in which the edges form a $t$-$(n,k,1)$
  design.  For a set $W$ of vertices from $H$, let $I_W$ denote the
  derived set of $W$.  The sub-hypergraph of $H$ induced by $I_W$
  is either a trivial hypergraph, or a $t$-$(|I_W|,k,1)$ design.
\end{theorem}
\begin{proof}
  If $|W| <t$, then $W$ cannot infect any edges of $H$. So the
  sub-hypergraph of $H$ induced by $I_W$ is empty.

  If $|W| = t$, then $W$ can infect only the edge that contains it.
  This single edge is, trivially, a $t$-$(k,k,1)$ design.

  Let $H'$ be the sub-hypergraph of $H$ induced by $I_W$.  Any
  $t$-subset $S$ of $I_W$ is contained in a unique edge $E$ of $H$. Since
  all the vertices of $S$ are infected, the set $S$ can infect
  the edge $E$. This means that $E$ is an edge of $H'$. So every
  $t$-subset of vertices from $I_W$ occur in an edge of $H'$. Since
  $H'$ is a sub-hypergraph of $H$, no $t$-subset can occur in more than
  one edge. Thus $H'$ is a $t$-$(|I_W|,k,1)$ design.
\end{proof}

\begin{cor}
Let $H$ be a hypergraph in which the edges form a $t$-$(v,k,1)$ design.
If the design does not contain any non-trivial sub-designs, then $I(H) =t+1$.
\end{cor}
\begin{proof}
  Let $W$ be any set of $t+1$ vertices that are not contained in a
  single edge. Then the derived set $I_W$ is larger than $k$. So the
  sub-hypergraph of $H$ induced by $I_W$ is neither a trivial
  hypergraph, nor a single edge. Thus the induced sub-hypergraph is a
  non-trivial $t$-$(|I_W|,k,1)$ sub-design; since the $t$-$(v,k,1)$ design
  has no non-trivial sub-design, it must include all edges of $H$. Thus
  $I_W$ includes all vertices of $H$ and $W$ is an infection set.
\end{proof}

Let $q$ be a prime power. Consider the finite projective space $PG(n,
q)$.  The points of $PG(n, q)$ are the $1$-dimensional subspaces of
the vector space $\bbF_q^{n+1}$, the lines are the $2$-dimensional
subspaces of the vector space $\bbF_q^{n+1}$; incidence between
points and lines is induced by incidence in the vector space
\cite[Section VI.7.5]{MR2246267}.  This space can be used to construct
a $2$-$(\frac{q^{n+1}-1}{q-1}, q+1, 1)$-design. For this we take the
points of $PG(n, q)$ as points of the design, and the lines of $PG(n,
q)$ as blocks.

\begin{proposition}
  Let $H$ be the hypergraph with edge set equal to the blocks of a
  $2$-$(\frac{q^{n+1}-1}{q-1}, q+1, 1)$-design constructed from the
  finite projective space $PG(n,q)$. Provided that $n \geq 2$, the
  infection number of $H$ is $n+1$.
\end{proposition}
\begin{proof}
  First note that Theorem~\ref{thm:derivedmakessubdesign} implies that
  the infection number of $H$ must be at least $n+1$. The derived set
  of any set of vertices will be a subspace of the projective space,
  so if the derived set includes all the vertices $t$ must contain a
  spanning set. Finally, any spanning set will form an infection set.
\end{proof}

\section{Hypergraph products}

In this section we will consider the behaviour of the infection number
over various hypergraph products.

We begin with the direct product.  For two hypergraphs $H_1$ and $H_2$
the \textsl{direct product} of $H_1$ and $H_2$ is the hypergraph with
vertex set $V(H_1) \cup V(H_2)$ and edge set
\[
\{ E_1 \cup E_2 \; | \; E_1 \in E(H_1), \, E_2 \in E(H_2)\}.
\]
The direct product of two hypergraphs is denoted by $H_1 \times H_2$.
The direct product can be defined recursively for any number of
hypergraphs. Note that the size of the edges increases, so this is not an extension
of the definition of the direct product of two graphs. 

We have already seen an example of the hypergraph direct product.  The
$k$-partite complete graph $H^{(k)}_{n_1,n_2,\dots, n_k}$ is
isomorphic to the direct product of the $k$ complete graphs
$H_{n_i}^{(1)}$ for $i =1, \dots ,k$. The case where $H_1$ is a
complete graph and $H_2$ has only a single vertex was considered in
Proposition~\ref{prop:addvertextocomplete}. To start, we will
generalize these results.

\begin{theorem}
  Let $k_i$ and $n_i$ be integers for $i = 1,\dots, \ell$ with $k_i
  \leq n_i$. If $n_i \geq 2k_i$  for at least one $i$, then 
\[
 I(H_{n_1}^{(k_1)} \times H_{n_2}^{(k_2)} \times  \dots \times H_{n_\ell}^{(k_\ell)} ) 
    = \sum_{i=1}^\ell (n_i - k_i).
\]
\end{theorem}
\begin{proof}
  Assume that $n_1 \geq 2k_1$, and we will construct an infection set $A$
  of size $\sum_{i=1}^\ell(n_i-k_i)$. For each graph $H_{n_i}^{(k_i)}$
  pick any $n_i -k_i$ vertices to be part of the set $A$; so $k_i$ of
  the vertices are uninfected in each sub-hypergraph. We show that $A$ is
  an infection set.

  Let $E$ be an edge that contains $k_1$ of the infected vertices from
  $H_{n_1}^{(k_1)}$ (since $n_1 \geq 2k_1$ this is possible) and all
  of the $k_i$ uninfected vertices from the other sub-hypergraphs
  $H_{n_i}^{(k_i)}$ with $i=2, \dots ,\ell$. Then the $k_1$ infected
  vertices from $H_{n_1}^{(k_1)}$ never occur with an uninfected
  vertex outside of this edge, so they can infect this edge. Then the only
  uninfected vertices remaining in the hypergraph are the $k_1$
  uninfected vertices from $H_{n_1}^{(k_1)}$. These vertices can be
  infected by any edge that contains them.

  This is optimal since there cannot initially be more
  than $k_i$ uninfected vertices in any sub-hypergraph $H_{n_i}^{(k_i)}$.
Thus 
\[
 I(H_{n_1}^{(k_1)} \times H_{n_2}^{(k_2)} \times  \dots \times H_{n_\ell}^{(k_\ell)} ) 
    \leq \sum_{i=1}^\ell (n_i - k_i).
\]
\end{proof}

Note that for the hypergraph in the previous theorem, $\sum n_i$ is
the number of vertices and $\sum k_i$ is the size of the edges. So
this theorem shows that the infection number of this product
hypergraph is one less than the upper bound given in
Proposition~\ref{prop:upperbound}. If it is not the case that one of the
sub-hypergraph $H_{n_i}^{(k_i)}$ has $n_i \geq 2k_i$, then the
infection number for the hypergraph is the upper bound.

\begin{lemma}
  Let $k_i$ and $n_i$ be integers for $i = 1,\dots, \ell$ with $k_i
  \leq n_i$. If $n_i < 2k_i$  for all $i$, then 
\[
 I(H_{n_1}^{(k_1)} \times H_{n_2}^{(k_2)} \times  \dots \times H_{n_\ell}^{(k_\ell)} ) 
    = 1 + \sum_{i=1}^\ell (n_i - k_i).
\]
\end{lemma}
\begin{proof}
 By Proposition~\ref{prop:upperbound}, there is an infection set of size
\[
n-k+1 = 1 + \sum_{i=1}^\ell (n_i - k_i).
\]
To show that the infection set cannot be any smaller, assume that $A$
is an infection set of size $\sum_{i=1}^\ell (n_i - k_i)$. Initially,
there cannot be more than $k_i$ uninfected vertices in any
sub-hypergraph $H_{n_i}^{(k_i)}$. So we can assume that $A$ contains
exactly $n_i - k_i$ vertices from each sub-hypergraph $H_{n_i}^{(k_i)}$.
Since each $n_i < 2k_i$, no sub-hypergraph has an edge containing only
infected vertices.

Let $A' \subset A$ be the subset of infected vertices that does the
first infection and assume it infects edge $E= E_1 \cup \dots \cup
E_\ell$. For any $i$, the size of $A' \cap E_i$ is no more than $n_i
-k_i <k_i$, so there is an uninfected vertex in $ H_{n_i}^{(k_i)}$. By
construction this vertex is adjacent to $A' \cap E_i$. This is a
contradiction.
\end{proof}

For any hypergraph $H$ the hypergraph $H \times H_1^1$ is the hypergraph
formed by adding a single vertex to every edge of $H$.  In some cases,
this operation can reduce the infection number of the hypergraph.
Consider the hypergraph $H$ on vertex set $\{1,2,3,4\}$ with edge set
$\{\{1,2\}, \{3,4\} \}$; this hypergraph has infection number two. Then $H
\times H_1^1$ is the hypergraph on five vertices with edge set $\{\{1,2,
5\}, \{3,4, 5\} \}$.  The infection number of  $H \times H_1^1$ is
one. Our next result proves that the infection number cannot drop by more
than one using this operation.

\begin{theorem}
For any hypergraph $H$
\[
I(H) -1 \leq I(H \times H_1^1) \leq I(H).
\]
Moreover, $I(H\times H_1^1)=I(H)-1$ if and only if $H$ is disconnected
and has at least one component which is a single edge.
\end{theorem}
\begin{proof}
  By definition, an infection set of $H$ is also an infection set
  of $(H\times H_1^1)$. As a result, it follows that $I(H\times
  H_1^1)\leq I(H)$.

  First we will show that if either $H$ is connected, or $H$ is
  disconnected with no component being a single edge, then
\[
I(H)\leq I(H\times H_1^1).
\]
Let $v$ be the vertex that is added to every edge in $H$ to form $H
\times H_1^1$.

If $H$ is a single edge $I(H\times H_1^1)$ and $I(H)$ both equal
$1$.  We assume that $H$ is not a single edge. Further, let $A$
be an infection set of $H\times H_1^1$ with cardinality $I(H\times
H_1^1)$. Since $H$ contains at least $2$ edges, $A \neq \{v\}$. In
fact, we will show that $A$ does not contain the vertex
$v$.

To the contrary, suppose that $v\in A$.  Let $A' \subset A$ be the set that
causes the first infection in the process; assume that it infects
the edge $E$ of $H\times H_1^1$.  By the infection rule, there is no
uninfected vertex $u$ outside of $E$ such that $\{u\}\cup A'$ is a
subset of another edge. In fact, noting that every edge of $H\times
H_1^1$ contains vertex $v$, we see that $\{u\}\cup (A'\setminus \{v\})$
is not contained in any edge other than $E$.  Thus, $A'\setminus \{v\}$
can infect $E$, and hence the vertex $v$.  This implies that $A
\setminus \{v\}$ is also an infection set of $H\times H_1^1$, which
contradicts the choice of $A$ being a minimum infection set.  Thus no
minimum infection set of $H\times H_1^1$ contains $v$.

Let $H_{i}$, for $i = 1,\dots,c$, be the components of $H$ and assume
that no component is a single edge (note that $c = 1$ if and only if
$H$ is connected). Consider the sub-hypergraph $H_i \times H_1^1$ in
$H \times H_1^1$. Let $B_i$ be an infection set for $H_i \times H_1^1$
with minimal size. Since no $H_i$ is a single edge, no $B_i$ contains
the vertex $v$. The infection set of $H\times H_1^1$ with the minimal
cardinality must contain a subset $B_i$ for each $i =1,\dots,c$.  Thus
$\bigcup_{i=1}^{c} B_{i}$ is an infection set of $H$, and
$I(H)\leq I(H\times H_1^1)$.

Next suppose that $H_i$ for $i = 1,\dots,\ell$ are the components of
$H$ that are a single edge.  For each $H_{i}$, with $i>\ell$, let
$B_{i}$ be an infection set of $H_i\times H_1^1$ with minimum
cardinality; from above this infection set does not contain $v$. 

The infection number of each $H_i$ with $i =1,\dots,\ell$ is $1$,
since each is a single edge. Further, $v$ can only infect one edge in
$(H\times H_1^1)$, since it is contained in every edge of the
hypergraph.  This implies that an infection set for $H \times H_1^1$
must contain at least $\ell-1$ vertices, other than $v$, to infect the
edges $H_i \times H_1^1$ for $i =1,\ldots, \ell$.  For $i=1, \dots,
\ell$, we select a single vertex $v_i$ from $H_i$.  Then
$(\bigcup_{i=1}^c B_{i}) \cup (\bigcup_{i=1}^{\ell-1} \{v_{i}\}) $ is
an infection set of $H\times H_1^1$ with minimal size.  But,
$(\bigcup_{i=1}^{n} B_{i}) \cup (\bigcup_{i=1}^{\ell-1} \{ v_{i} \})$
is an infection set of $H$. This implies that $I(H)\leq I(H\times
H_1^1)+1$.

Conversely, note that each infection set of $H$ with cardinality
$I(H)$ contains exactly one vertex from each single edge.  Thus $I(H)
= | (\bigcup_{i=1}^c B_{i})\cup (\bigcup_{i=1}^{\ell}
\{v_{i}\}) |$, (where $B_{i}$ is the infection set of $H_{i}$ for $i
>\ell$, and $v_i$ is any vertex in $H_i$ with $i = 1, \dots,
\ell$). Exactly one of the vertices $v_i$ may be removed from this
infection set for $H$ to form an infection set for $H\times H_1^1$
(since the $v$ can infect the edge $H_i$).
\end{proof}

There is a simple bound on the infection number for the direct product
of any two hypergraphs.

\begin{lemma}
Let $H_1$ and $H_2$ be two hypergraphs, then 
\[
I(H_1 \times H_2) \leq I(H_1) + I(H_2).
\]
\end{lemma}
\begin{proof}
  If $A$ is an infection set for $H_1$, and $B$ is an infection set for
  $H_2$, then $A \cup B$ is an infection set for $H_1 \times H_2$.

  To see this assume that $A_1 \subset A$ infects the edge $E$, and
  $B_1 \subset B$ infects the edge $F$. Then $A_1 \cup B_1$ is a
  subset of infected vertices in $E \cup F$. We claim that $A_1 \cup
  B_1$ can infect $E \cup F$.  Assume that there is an uninfected
  vertex $v$, outside of $E \cup F$, that is adjacent to $A_1 \cup
  B_1$. If $v$ is a vertex from $H_1$, then $A_1 \cup \{v\}$ is
  contained an edge other than $E$. This contradicts the fact that
  $A_1$ infects $E$. Similarly, $v$ cannot be a vertex of $H_2$, thus
  no such vertex $v$ can exist.
\end{proof}

The next example shows that this bound can hold with equality.

\begin{proposition}
  If $H_1$ and $H_2$ are both hypergraphs with more than one edge and
  $I(H_1) =1$ and $I(H_2) =1$, then $I(H_1 \times H_2) = 2$.
\end{proposition}
\begin{proof}
  Since both $H_1$ and $H_2$ have more than one edge, no vertex of
  $H_1 \times H_2$ has degree one, so by Lemma~\ref{lem:degreebound},
  $I(H_1 \times H_2) >1$. The previous lemma result then gives the
  result.
\end{proof}

Next we define two additional constructions of hypergraphs that are
based on the corona of a graph. The \textsl{join} of a graph $G$ with
a vertex $v$, is the graph with vertex set $V(G) \cup \{v\}$, and the
edge set is the set of all the edges from $G$ along with all edges of
the form $\{v,h\}$ where $h \in V(G)$. So it is the graph formed by
adding $v$ to $G$ and making every vertex in $G$ adjacent to $v$. If
$G$ and $H$ are graphs, the \textsl{$H$-corona} of $G$ is the graph
formed by taking $G$ and for each vertex in $G$ joining a copy of $H$
to the vertex. The $H$-corona of $G$ is denoted by $G \circ
H$. In~\cite{MR2388646} an upper bound on the zero forcing number for
the corona of a graph is given; this bound follows from a simple
construction of a zero forcing set for $G \circ H$ using zero forcing
sets of $G$ and $H$.

\begin{theorem}\label{Zcorona} ~\cite[Prop. 2.12]{MR2388646}
Let $G$ and $H$ be graphs, then
\[
Z( G \circ H) \leq Z(G) \, |V(H)| + (|V(G)| - Z(G))\, Z(H). 
\]
\end{theorem}

This first generalization of a corona to hypergraphs we will call the
weak corona. Let $H$ be a $(k-1)$-hypergraph. Define the join of $H$
with a vertex $\{v\}$ to be the hypergraph formed by adding $v$ to
every edge of $H$.  Define the \textsl{weak $H$-corona} of a
$k$-hypergraph $G$ to be the $k$-hypergraph formed by taking $G$ and joining a copy
of $H$ to each vertex of $G$. This new hypergraph is denoted by $G
\circ_w H$.  If $H$ is a hypergraph with $\ell$ disjoint edges, then
the $H$-corona of a hypergraph $G$ is an example of an $\ell$-corona
of $G$ as defined in~\cite{tuza}. Figure~\ref{exampleofweakcorona} is
the weak corona of a cycle with a single edge on two vertices.

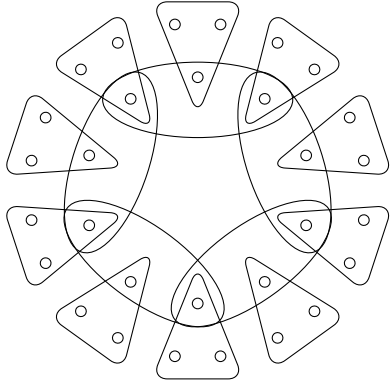
\begin{figure}[h!]
\begin{center}
\begin{tikzpicture}
\foreach \x in {0,1,...,9} \draw [rotate=36*\x]   \mytriangle;
\foreach \x in {0,1,...,4} \draw [rotate=72*\x]   \myellipse;
\end{tikzpicture}
\end{center}
\caption{The weak corona of a cycle with an edge} \label{exampleofweakcorona}
\end{figure}

Using a construction of an infection set based on the construction
used in the proof of Theorem~\ref{Zcorona}, we can establish a similar
bound.

\begin{theorem}\label{Zcorona-2}
Let $G$ and $H$ be hypergraphs, then
\[
I( G \circ_w H) \leq |V(G)| \, I(H). 
\]
\end{theorem}
\begin{proof}
  In $G \circ_w H$, there is a copy of $H$ for each vertex $g$ in $G$;
  call this sub-hypergraph $H_g$.  The vertices of $G \circ_w H$ are
  either $g\in V(G)$, or $h_g \in V(H_g)$.

  Let $A$ be an infection set for $H$ and let $A_g$ be the copy of $A$ in
  the hypergraph $H_g$. Define $B = \bigcup\limits_{g \in G} A_g$. We claim that
  $B$ is an infection set for $G \circ_w H$.

  If $A' \subset A$ infects an edge $E$ in $H$, then $A'_g \subset
  A_g$ (the set of vertices corresponding to $A'$ in $H_g$) can infect
  the edge $E_g$ in $H_g$ corresponding to $E$. The edge $E_g$
  contains $g$, so this vertex will also be infected. Thus following
  the infection process for $H$, the set $A_g$ infects all the
  vertices in $H_g$ (including $g$) for all the $g \in
  V(G)$. Therefore the set $A$ is an infection set for $G \circ_w H$.
\end{proof}

We will consider another generalization of the corona to hypergraphs.
Let $H$ be a $k$-hypergraph. Define the \textsl{strong join} of $H$
with a vertex $v$ to be the hypergraph with vertex set $V(H) \cup
\{v\}$, and edge set containing all edges of $H$, together with all edges
of the form $F \cup \{v\}$ where $F$ is a $(k-1)$-subset of an edge in
$H$. Denote this graph by $H \triangledown v$.

Let $H$ be a $k$-hypergraph. The \textsl{strong $H$-corona} of a
$k$-hypergraph $G$ is defined to be the $k$-hypergraph formed by
taking first a copy of $G$, and then for each vertex $g \in V(G)$ taking
a disjoint copy of $H$ and $H \triangledown g$.  This hypergraph is denoted by
$G \circ_s H$. The number of vertices in $G \circ_s H$ is $|V(G)| +
|V(G)|\,|V(H)|$ and the number of edges is
\[
|E(G)| + |V(G)| \, |E(H)| (k+1).
\]

\begin{theorem}\label{strongcorona}
Let $G$ and $H$ be hypergraphs, then
\[
I( G \circ_w H) \leq I(G) \, |V(H)| + (|V(G)| - I(G) )\, I(H).
\]
\end{theorem}
\begin{proof}
  For each vertex $g \in V(G)$, let $H_g$ denote the copy of $H$
  corresponding to $g$.

  Construct an infection set $A$ for $G \circ_s H$ as follows. 
  Let $B$ be a minimal infection set for $G$.  For every $g \in B$,
  set all the vertices in $H_g$, except $g$, to be in $A$ (so
  initially infected). For every vertex $g$ of $G$ not in $B$, add a
  minimal infection set from $H_g$ to $A$. 

  For every $g \in B$, the vertices of $H_g$ are all infected, so they
  can infect $g$. Then the vertices $g \in B$ from the copy of $G$ in
  $G\circ_w H$ can infect some vertices in $G$. Assume that $g_1$ is one of
  these vertices infected in the first step of the infection process
  for $B$. Then $g_1$ is initially uninfected, and $A$ includes an
  infection set from $H_{g_1}$. Once $g_1$ is infected, the infection
  set from $H_{g_1}$ can infect the vertices in $H_{g_1}$. Now $g_1$
  can infect other vertices in the copy of $G$. Continuing
  like this, all of the vertices in $G \circ_w H$ will be infected.
 \end{proof}

 The next product that we consider is the \textsl{Cartesian
   product}. If $H$ and $G$ are hypergraphs, $G \Box H$ is the
 Cartesian product.  The vertex set of this hypergraph is $V(G) \times
 V(H)$. For a set $E=\{e_1,e_2,\dots,e_k\}$ and a vertex $v$, define 
\[
E \times v = \{ (e_1,v), (e_2,v), \dots, (e_k,v)\},
\quad
v \times E = \{ (v, e_1), (v, e_2), \dots, (v, e_k)\}
\]
The edge set of the Cartesian product is
\[
\{ E \times v : E \in E(G), \, v \in V(H)\} \cup \{ u \times E: u \in V(G), \, E \in E(H)\}.
\]

There is a well-known simple upper bound for the zero forcing number of the
Cartesian product of two graphs (see \cite{MR2388646});
\begin{align}
Z(G \Box H) \leq Z(G) \, |V(H)|. \label{eq:bnd_cartesian_prod_graphs}
\end{align}

We generalize this result for the infection number of the Cartesian
product to hypergraphs. For this we first introduce a generalization of
the infection number. Consider the following rule.

{\bf $m$-Infection Rule:} A set $A$, with $|A| \geq m$, of infected vertices
can infect the vertices in an edge $E$ if:
\begin{enumerate}
\item $A \subset E$, and 
\item there are no uninfected vertices $v$, not contained in $E$, such that $A
  \cup \{v\}$ is a subset of an edge.
\end{enumerate}

For $m=1$ this corresponds to our normal infection rule. A set of
vertices of a hypergraph is an \textsl{$m$-infection set} when the
vertices of the set are initially infected and applying the
$m$-infection rule repeatedly infects all vertices of the hypergraph.
The \textsl{$m$-infection number} of a hypergraph is the size of a
smallest infection set for $H$. The $m$-infection number is denoted by
$I_m(H)$. It is clear that $I_m(H) \leq I_{m+1}(H)$. Similarly,
$I_{m+1}(H)$ can be bounded in terms of $I_{m}(H)$.

\begin{lemma}\label{lem:basic_bnd_minfection}
  Let $H$ be a hypergraph. Then
  \begin{align*}
    I_{m+1}(H) \leq I_{m} + |E(H)|.
  \end{align*}
\end{lemma}
\begin{proof}
  Let $A$ be an $m$-infection set. Let $\cE$ be the edges of $E(H)$,
  which are not completely contained in $A$.  Let $f$ be a function
  from $\cE$ to $V(H)$ that maps $E \in \cE$ to one element of $E
  \setminus A$. 

  Let $B$ denote the set $A \cup \{ f(E): E \in \cE \}$. We claim that
  $B$ is an $(m+1)$-infection set.  Clearly $B$ infects all the edges
  of $H$ that are not in $\cE$.  Since $A$ is an $m$-infection set,
  there is a set $A' \subseteq A$ of size $m$ and an edge $E \in \cE$
  such that $A'$ can $m$-infect $E$.  Hence, $E$ is the only edge
  containing $A' \cup \{ f(E) \}$, so $A' \cup \{ f(E) \}$ can
  $(m+1)$-infect $E$. Continuing like this (following the
  $m$-infection process with $A$), $B$ is an $(m+1)$-infection set. As
  $|\cE| \leq |E(H)|$, this shows the assertion.
\end{proof}

This bound can be tight, for example if $H$ consists of one single
edge of size at least $2$.

\begin{theorem}\label{prop:bnd_cartesian_prod}
  Let $G$ and $H$ be hypergraphs. Then
  \begin{align*}
    I( G \Box H ) \leq I(G) I_2(H).
  \end{align*}
\end{theorem}
\begin{proof}
  Let $A_G$ be an infection set of $G$ and $A_H$ a $2$-infection set of $H$.
  We define a set 
\[
A = \{ (a_G, a_H): a_G \in A_G, a_H \in A_H \},
\]
we will show that $A$ is an infection set of $G \Box H$.  To start, define
\[
A_1 = \{ (a_G, a_H): a_G \in A_G, a_H \in V(H) \}.
\]
We will show that $A$ infects all the vertices of $A_1$.

Let $E_H$ be an edge of $H$. In some step of the $2$-infection process
of $H$, the edge $E_H$ is $2$-infected by a set $A_H'$ (starting from
$A_H$). We want to show that for a given $v \in A_G$, the edge $v
\times E_H$ is the only edge of $G \Box H$ that contains $\{ v \}
\times A_H'$. Suppose that there is an edge $E$ of $G \Box H$ that
contains $\{ v \} \times A_H'$.  The set $A_H'$ has size at least $2$,
so $E$ contains vertices $(v, w_1)$ and $(v, w_2)$ for some $w_1, w_2
\in V(H)$. Hence, $E$ has the form $v \times E_H'$ for some edge
$E_H'$ of $H$.  As $A_H'$ is a $2$-infection set, $E_H' = E_H$, and
our first claim follows.
  
Next we will show that $A_1$ infects all the vertices $(g,h)$ where $g
\not \in A_G$. Let $E_G$ be an edge of $G$ that is infected by a set
$A_G'$ (starting from $A_G$). We want to show that for a given $w \in
V(H)$, the edge $E_G \times w$ is the only uninfected edge of $G \Box
H$ that contains $A_G' \times \{ w \}$. If $A_G'$ has size at least
$2$, then this can be seen since it is similar to the first claim. If
$A_G'$ has size $1$, then we can assume that $(v,w)$ is the single
element of $A_G' \times \{ w \}$ (where $v \in V(G)$). The only edges
of $G \Box H$ containing $(v, w)$, other than $E_H \times w$, must
have the form $v \times E_H$ for some edge $E_H$ of $H$. By our first
claim, the edge $E_H$ is already infected.  This proves the assertion.
\end{proof}

Notice that we have $I_2(H) = V(H)$ if $H$ is a graph, so our result
implies Equation \eqref{eq:bnd_cartesian_prod_graphs}. Lemma
\ref{lem:basic_bnd_minfection} and Theorem
\ref{prop:bnd_cartesian_prod} imply the following bound that avoids
the concept.

\begin{corollary}
Let $G$ and $H$ be two hypergraphs, then
\[
I(G \Box H) \leq I(G) (I(H) + |E(H)|).
\]
\end{corollary}

This bound can be tight, e.g. when $G$ and $H$ each consist of one single edge of size at least $2$.

\section{Line graphs}
\label{sec:linegraph}

In Section~\ref{sec:defn} the line graph of a hypergraph is defined.
It is possible to construct a hypergraph $H$ for any graph $G$ such
that $L(H) = G$. Simply let the vertices of $H$ be the edges in $G$
and for each vertex of $G$ let the set of all edges incident to the
vertex be an edge in $H$.  Call this the \textsl{adjacency hypergraph for $G$}.

\begin{lemma}
  For any connected simple graph, the infection number of the corresponding
  adjacency hypergraph is $2$.
\end{lemma}
\begin{proof}
  Let $G$ be a connected simple graph and $H$ the adjacency hypergraph
  of $G$. Each vertex in $H$ has degree $2$ since it corresponds to an
  edge in $G$ with two end points. Further a pair of vertices in $H$
  can belong to at most one edge, since otherwise they would
  corresponds to a double edge in $G$.

  Thus any two adjacent infected vertices in $H$ can infect the edge
  in which they are contained. Then the remaining vertices in the edge
  occur in only one other edge each. These vertices can then force the
  other edges that contain them. Since $G$ is connected, continuing
  like this the hypergraph $H$ can be infected by any two adjacent
  vertices.
\end{proof}

It has been shown in~\cite[Theorem 3.11]{eroh} that for any graph $G$ 
\[
Z(G) \leq 2 Z( L(G) ).
\]

We state the following lemma, but we suspect that a stronger result holds.
\begin{lemma}
For any $k$-hypergraph $H$,
\[
I(H) \leq k Z( L(H) ).
\]
\end{lemma}
\begin{proof}
  For $v$ in the vertex set of $L(H)$, let $E_v$ be the corresponding
  edge in $H$. Let $Z$ be a zero forcing set for $L(H)$. Add all the
  vertices of $H$ that are in the edge $E_v$ where $v \in Z$ to a set
  $A$. We claim that $A$ is an infection set for $I(H)$ of size
  $k\,Z(L(H))$.

  Assume that in the first step of the zero forcing process on $L(H)$
  (using the set $Z$) that $v$ forces $w$. Then all the vertices in
  $E_v$ are initially infected, further, all the vertices in any edge
  intersecting $E_v$ are also all infected, except the vertices in
  $E_w$. Thus the vertices in $E_v \cap E_w$ can infect
  $E_w$. Continuing like this, shows that $A$ is an infection set for
  $I(H)$.
\end{proof}

\section{Further work}

This paper is a first generalization of zero forcing to
hypergraphs. There are many open problems and different directions in
which to continue this work. A few are listed below.

We have seen that there are many different hypergraphs with infection
number $1$, including any hypergraph in which the vertex set is an
edge.  It is also clear that if the infection number of a reduced
hypergraph is $1$, then that hypergraph must have a degree $1$
vertex. But this is in no way a sufficient condition for a hypergraph
to have infection number $1$. It would be interesting to find some
characterization of the $k$-hypergraphs that have infection number
equal to $1$.

Similarly, a graph on $n$ vertices has zero forcing number $n-1$ if
and only if it is a complete graph. We would like to be able to
characterize the $k$-uniform hypergraphs on $n$ vertices with
infection number $n-(k-1)$.  To this end, we make the following
conjecture.

\begin{conjecture}
  Assume that $H$ is a $k$-uniform hypergraph on $n$ vertices with
  $I(H) =n-(k-1)$.  Then for any infection set of size $n-(k-1)$, the
  set of $k-1$ uninfected vertices are contained in an edge of $H$.
\end{conjecture}

The zero forcing number for a tree has been well-studied~\cite{MR2388646}. In
fact, it has been shown that the zero forcing number of a tree is
equal to the minimum number of paths needed to cover the edges of the
tree. It would be interesting to have an analogous result for
hypertrees. 

The first step is to determine what the appropriate hypergraph version of a tree
is; one definition uses the notion of a \textsl{host graph}. For a
hypergraph $H$, a host graph is a graph on the same vertex set as $H$
with the property that the vertices in any edge of $H$ induce a
connected subgraph in the host graph.  A hypergraph is called
a \textsl{hypertree} if there exists a host graph of the hypergraph
that is a tree.

So far our results in this direction are unsatisfying.
For example, if $H$ is a hypertree, then it may happen that $I(H) > Z(T)$ for a
host tree of $H$. To see this consider the hypergraph 
\[
H:=\{ (1,5,6),\, (2,5,6),\, (3,5,6),\, (4,5,6)\}.
\]
The tree below is a host tree for $H$
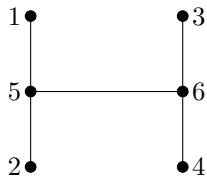
\begin{figure}[h!]
\begin{center}
\begin{tikzpicture}
\draw (-1,1) -- (-1,0) -- (-1,-1);
\draw (1,1) -- (1,0) -- (1,-1);
\draw (1,0) -- (-1,0);
\draw [fill] (-1,0)  circle (2pt) node [left] {5};
\draw [fill] (1,0)  circle (2pt) node  [right] {6};
\draw [fill] (-1,1)  circle (2pt) node  [left] {1};
\draw [fill] (-1,-1)  circle (2pt) node  [left]{2};
\draw [fill] (1,1)  circle (2pt) node  [right] {3};
\draw [fill] (1,-1)  circle (2pt) node  [right] {4};
\end{tikzpicture}
\end{center}
\end{figure}

Note that $I(H) = 3$; the set $\{1,2,3\}$ can infect the hypergraph,
and there is no smaller set that forms an infection set. Also, the
zero forcing number for the host tree above is $2$, since $\{1,3\}$ is
a zero forcing set.  In fact, we can make this difference arbitrarily
large. Let $H = \{E_1, E_2, \dots, E_p\}$ be a hypergraph on $n$
vertices that is a flower with $|E_i \cap E_j|=1$, then the graph
$K_{1,n-1}$ is a host tree for $H$. In this case,
\[
p-1 = I(H) < Z(K_{1,n-1}) = n-2.
\]
The graph formed by taking the vertex sum of $p$ paths $P_i$, each
with length $|E_i|-1$ is also a host graph for $H$. The zero forcing
number of this graph is $p-1$.

We propose two questions about hypertrees. First, is it possible to determine a formula for the
infection number of a hypertree or for a subfamily of hypertrees?
Second, is there a relationship between zero forcing of a host tree and
infection number of hypertree?

If $2t+1 \geq k$ and $n\leq 2t$, then we are not  aware of a formula for the
infection number for $C^{(k)}_{n}(t)$. We would like have a more
complete picture for infection of hypercycles.

Finally, we have defined $m$-infection, this concept is useful for
Cartesian products. It would be interesting to determine the
$m$-infection number for some families of graphs and to find other
applications of the concept.

{\bf Acknowledgments.} The work in this paper was a joint project of
the Discrete Mathematics Research Group at the University of Regina,
attended by all the authors. 

\bibliographystyle{plain}
\bibliography{infection_literature}

\end{document}